\documentclass[11pt]{article}

\newcommand{\reals}{\mbox{$\mathbb R$}}

\newcommand{\nats}{\mbox{$\mathbb N$}}
\newcommand{\ints}{\mbox{$\mathbb Z$}}

\newcommand{\comment}[1]{}

\usepackage{amsmath, amsthm, amsfonts, amssymb, latexsym, stmaryrd}

\comment{

\def\squarebox#1{\hbox to #1{\hfill\vbox to #1{\vfill}}}
\def\qed{\hspace*{\fill}
        \vbox{\hrule\hbox{\vrule\squarebox{.667em}\vrule}\hrule}\smallskip}
\newenvironment{proof}{\begin{trivlist}
  \item[\hspace{\labelsep}{\em\noindent Proof.~}]
  }{\qed\end{trivlist}}
}

\newtheorem{lemma}{Lemma}[section]
\newtheorem{theorem}[lemma]{Theorem}
\newtheorem{corollary}[lemma]{Corollary}
\newtheorem{proposition}[lemma]{Proposition}
\newtheorem{claim}[lemma]{Claim}
\newtheorem{observation}[lemma]{Observation}
\newtheorem{definition}[lemma]{Definition}

\def\squareforqed{\hbox{\rlap{$\sqcap$}$\sqcup$}}
\def\qed{\ifmmode\squareforqed\else{\unskip\nobreak\hfil
\penalty50\hskip1em\null\nobreak\hfil\squareforqed
\parfillskip=0pt\finalhyphendemerits=0\endgraf}\fi}

\newlength{\tablength}
\newlength{\spacelength}
\newcommand{\tabstar}{\hspace*{\tablength}}
\newcommand{\spacestar}{\hspace*{\spacelength}}
\def\obeytabs{\catcode`\^^I=\active}
{\obeytabs\global\let^^I=\tabstar}
{\obeyspaces\global\let =\spacestar}
\newenvironment{display}{\begingroup\obeylines\obeyspaces\obeytabs}{\endgroup}
\newenvironment{prog}{\begin{display}\parskip0pt\sf}{\end{display}}

\title{Extremal subgraphs of the $d$-dimensional grid graph}

\author{
{\sl Geir Agnarsson}
\thanks{Department of Mathematical Sciences,
George Mason University, 
MS 3F2, 
4400 University Drive, 
Fairfax, VA -- 22030, 
{\tt geir@math.gmu.edu}}
\and
{\sl Kshitij Lauria}
\thanks{Computer Science Department,
Brown University,
Box 1910, 
115 Waterman St., 4th Floor
Providence, RI -- 02912,
{\tt klauria@cs.brown.edu}} 
} 

\date{}

\begin{document}

\maketitle

\begin{abstract}
For each natural number $n$ we determine, both
asymptotically and exactly, the maximum number
of edges an induced subgraph of order $n$ of 
the $d$-dimension a grid graph ${\ints}^d$ can have.
The asymptotic bound is obtained by using a theorem
Bollob\'{a}s and Thomason, and the exact bound is obtained by
induction. This generalizes some earlier results for the case $d=2$
on one hand, and for $n\leq 2^d$ on the other.

\vspace{3 mm}


\vspace{2 mm}

\end{abstract}

\section{Introduction}
\label{sec:intro}

The purpose of this article is to determine the maximum
number of edges an induced subgraph on $n$ vertices 
of the $d$-dimensional rectangular grid graph ${\ints}^d$
can have. The very first non-trivial result in an exact manner
for the case $d=2$ appears in~\cite{Harary-Harborth},
where it is shown that this maximum number of edges
is given by $\lfloor 2n - 2\sqrt{n}\rfloor$. Some 
other interesting and related exact results appear in~\cite{Brass},
where the author Peter Bra{\ss} studies $f(n,k)$, the maximum
number of unit distances among $n$ points in the plane, where
the additional restriction is added that only those unit distances
are counted that are among a fixed set of $k$ directions. Here
the maximum is taken over all sets of $n$ points and all sets
of $k$ directions. The case $k=1$ is trivial, whereas for the 
case $k=2$ it suffices to consider subgraphs of ${\ints}^2$,
and so it coincides with the mentioned result from~\cite{Harary-Harborth},
and so $f(n,2) = \lfloor 2n - 2\sqrt{n}\rfloor$. Other values of $f(n,k)$ have
not yet been determined exactly.
 
In this paper we assume $d$ to be fixed and
$n$ an unrestricted positive integer variable. Note that when $n\leq 2^d$, 
the problem reduces to determine the maximum number
of edges of an induced subgraph on $n$ vertices 
of the $d$-dimensional hypercube $Q_d$, a study already
done in part in the 1970's as described in~\cite{Geir-Hcube}.
We will briefly revisit this case in the last Section~\ref{sec:some-obs}.
In~\cite{Geir-Hcube} a recap and references of know results regarding the 
case $n\leq 2^d$ and induced subgraphs of the hypercube are presented.

The considerations in this article were in part initially inspired by 
the heuristic integer sequence 
$0,1,2,4,5,7,9,12,13,15,17,20,21,23,25,\ldots$~\cite[A007818]{wrong}, 
describing the maximal number of edges joining $n = 1, 2, 3, \ldots$
vertices in the cubic rectangular grid ${\ints}^3$, for which no general
formula nor procedure to compute it is given. -- First we set forth 
our basic terminology and definitions.

\paragraph{Notation and terminology}
The set of integers will be denoted by $\ints$,
the set of natural numbers $\{1,2,3,\ldots\}$ by $\nats$, 
and the set $\{1,2,3,\ldots,n\}$ by $[n]$.
Unless otherwise stated, all graphs in this article
will be finite, simple and undirected. For a graph $G$, its
set of vertices will be denoted by $V(G)$ and its set
of edges by $E(G)$. Clearly $E(G)\subseteq \binom{V(G)}{2}$,
the set of all 2-element subsets of $V(G)$. We will sometimes denote
an edge with endvertices $u$ and $v$ by $uv$ instead of the
actual 2-set $\{u,v\}$. The {\em order}
of $G$ is $|V(G)|$ and the {\em size} of $G$ is $|E(G)|$.
By an {\em induced subgraph} $H$ of $G$ we mean a subgraph
$H$ such that $V(H)\subseteq V(G)$ in the usual set theoretic
sense, and such that if $u,v\in V(H)$ and $uv\in E(G)$, then
$uv\in E(H)$. If $U\subseteq V(G)$ then the subgraph of $G$ 
induced by $U$ will be denoted by $G[U]$. For $d\in\nats$, a 
{\em rectangular grid}
${\ints}^d$ in our context is a infinite graph with the point set
${\ints}^d$ as its vertices and where two points
$\tilde{x} = (x_1,\ldots,x_d)$ and $\tilde{y} = (y_1,\ldots,y_d)$
are connected by an edge iff the {\em Manhattan distance}
$D(\tilde{x},\tilde{y}) = \sum_{i=1}^d |x_i - y_i| = 1$. We will
talk about ${\ints}^d$ both as a point set and an infinite
graph. Most of the times we will restrict ourselves to the subgraph 
${\nats}^d$ of ${\ints}^d$. 
For $I\subseteq [d]$ let $\pi_I : {\ints}^d \rightarrow {\ints}^{|I|}$
be the projection where all the coordinates of $\tilde{x}$ that are not
in $I$ are omitted. Many times we will omit the set-brackets
and simply list the present coordinates; for example, 
$\pi_{\hat{\imath}}$ and $\pi_{1,\ldots,i-1,i+1,\ldots,d}$ will mean the projection
$\pi_{\{1,\ldots,i-1,i+1,\ldots,d\}}$ onto ${\ints}^{d-1}$.
Adopting this notation for sets we will denote the complement
of $I\subseteq [d]$ by either $\overline{I}$ or $\widehat{I}$,
when emphasizing the fact that we are omitting elements from
the index set $I$. 

\paragraph{Organization of article}
The rest of this article is organized as follows.
In Section~\ref{sec:f-nested} we show that when 
considering the maximum number of edges a finite set
$S\subseteq {\ints}^d$ of a given order can induce, we can assume that
the slices of $S$ perpendicular to each axis are nested,
in the sorted ``Tower of Hanoi'' fashion, or more precisely,
we can assume $S$ to be ``fully nested'' in the sense of 
Definition~\ref{def:nested} here below.

In Section~\ref{sec:asymp} we use a result by
Bollob\'{a}s and Thomason~\cite{Bollobas-Thomason} to
obtain a tight asymptotic upper bound for the maximum number
$E_d(n)$ of edges a set $S\subseteq {\ints}^d$ with $|S| = n$
can induce, for fixed $d$ and $n$.

In Section~\ref{sec:rec-ineq} we derive an important
recursive inequality for $E_d(n)$ as stated in Lemma~\ref{lmm:rec-ineq}.

In Section~\ref{sec:cubics} we 
introduce a specific class of fully nested sets, {\em $d$-cubicles}
$\llbracket n\rrbracket^d$ for each $n\in\nats$, and 
start verifying that these $d$-cubicles in ${\nats}^d$ are sets
that have exactly the maximum $E_d(n)$ number of induced edges.
This will be done by induction on $n+d$.

In Section~\ref{sec:E=F} we prove some important properties
of the $d$-cubicles $\llbracket n\rrbracket^d$ and complete
the inductive argument from previous Section~\ref{sec:cubics}
and obtain an exact formula for $E_d(n)$, as stated in 
Theorem~\ref{thm:main}, the main theorem of this paper. 

In the final Section~\ref{sec:some-obs} we state
some corollaries we obtain from Theorem~\ref{thm:main}
when we consider some special cases, which have been
derived and reported in the literature. Hence, 
we point out how Theorem~\ref{thm:main} is a generalization
of some celebrated known results. 

\section{Fully nested sets}
\label{sec:f-nested}

Let $S\subseteq {\ints}^d$ of order $n$. We denote
the number of edges of the induced subgraph $G[S]$
of ${\ints}^d$ by $E_d(S)$, and let
\begin{equation}
\label{eqn:defE(n)}
E_d(n) = \max\{E_d(S): S\subseteq {\ints}^d, \ |S| = n\}
\end{equation}
be the maximum number of edges $G[S]$ can have.
The main objective of this article is to determine
$E_d(n)$, both asymptotically and exactly.
By translation, and without loss of generality,  
we may assume that $S\subseteq {\nats}^d$. 
For $i\in [d]$ let $g_i$ be the {\em gravity along $i$-th axis} 
that acts on $S$ in the following way: For each 
$\tilde{y}\in \pi_{\hat{\imath}}(S)$ order the elements of 
$S\cap \pi_{\hat{\imath}}^{-1}(\tilde{y})$ 
linearly by their $i$-coordinate, say 
$\tilde{x}_1,\tilde{x}_2,\tilde{x}_3,\ldots$, and then replace
the $i$-th coordinate $\pi_i(\tilde{x}_h)$ 
in each $\tilde{x}_h$ by its placement $h$ in this linear
order, thereby obtaining the set
$g_i(S\cap \pi_{\hat{\imath}}^{-1}(\tilde{y}))$ that induces
a path on $|S\cap \pi_{\hat{\imath}}^{-1}(\tilde{y})|$ points
in ${\nats}^d$ parallel to the $i$-th axis. From the partition 
$S = \bigcup_{\tilde{y}\in\pi_{\hat{\imath}}(S)}
(S\cap \pi_{\hat{\imath}}^{-1}(\tilde{y}))$, we let
\[
g_i(S) :=
\bigcup_{\tilde{y}\in\pi_{\hat{\imath}}(S)}
g_i(S\cap \pi_{\hat{\imath}}^{-1}(\tilde{y})).
\]
One can imagine the points of $S$ represented as cube-like blocks
in zero-gravity $d$-space, with the edges of the blocks parallel
to the axes, and $g_i(S)$ the location of the blocks
after the gravity $g_i$ pulls the set $S$ of blocks down towards 
the hyperplane with $i$-th coordinate equal $1$. 

Clearly we have $g_i^2(S) = g_i(g_i(S)) = g_i(S)$.
Note that in general $g_i(g_j(S))\neq g_j(g_j(S))$, so
the operators $g_1,\ldots, g_d$ do not commute. We let 
$g := g_1g_2\cdots g_{d-1}g_d$ be the {\em total gravity}
acting on the set $S$.
\begin{definition}
\label{def:nested}
A set $S\subseteq {\nats}^d$ is {\em $i$-nested} if
\[
\pi_{\hat{\imath}}(S\cap \pi_i^{-1}(1))\supseteq 
\pi_{\hat{\imath}}(S\cap \pi_i^{-1}(2))\supseteq 
\pi_{\hat{\imath}}(S\cap \pi_i^{-1}(3))\supseteq\cdots.
\]
The set $S$ is {\em fully nested} of it is $i$-nested
for each $i\in[d]$.
\end{definition}
Note that $g_i(S) = S$ holds iff for each $\tilde{y}\in\pi_{\hat{\imath}}(S)$
we have $\pi_i(S\cap\pi_{\hat{\imath}}^{-1}(\tilde{y})) = [k]$ where
$k = |S\cap\pi_{\hat{\imath}}^{-1}(\tilde{y})|$. From this,
and the mere definition of the inverse image in general,  
we see that for a set $S\subseteq {\nats}^d$ we then have the
following.
\begin{observation}
\label{obs:gi=inest}
$S$ is $i$-nested iff $g_i(S) = S$.
\end{observation}
By symmetry we note that if $p : {\nats}^d \rightarrow {\nats}^d$
is the linear map that permutes the coordinates and
$S\subseteq {\nats}^d$ is a set, then $S$ is fully nested
iff $p(S)$ is fully nested.

Our next objective is to prove the following theorem,
the statement of which is seemingly obvious for 
dimensions $d\leq 3$.
\begin{theorem}
\label{thm:fully-nested}
For $S\subseteq {\nats}^d$ we have 
that $g(S) = S$ iff $g_i(S) = S$ for each $i\in [d]$, that is
$S$ is $i$-nested for each $i$.
\end{theorem}
Clearly, if $S$ is $i$-nested for each $i$, then
$g(S) = S$. To verify the other implication, we 
need a couple of lemmas.
\begin{claim}
\label{clm:2-chains}
Let $a_1,\ldots,a_n$ and $b_1,\ldots,b_n$ 
be two strings of real numbers with $a_i\geq b_i$ for each $i$. If 
$a_1'\geq a_2'\geq\cdots\geq a_n'$ and
$b_1'\geq b_2'\geq\cdots\geq b_n'$ are
sortings of these strings in descending order,
then for each $i=1,\ldots,n$ we have $a_i'\geq b_i'$.
\end{claim}
\begin{proof}
We may assume that the $a_i$s are already
ordered $a_1\geq a_2\geq\cdots\geq a_n$.
By assumption we then have $a_i\geq b_i, a_{i+1},\ldots, a_n$
and hence $a_i\geq b_i,b_{i+1},\ldots, b_n$. Since the $b_i$'s 
are sorted descendingly $b_1'\geq b_2'\geq\cdots\geq b_n'$, 
we have in particular $a_i\geq b_i',b_{i+1}',\ldots, b_n'$.
\end{proof}
A direct consequence of Claim~\ref{clm:2-chains} is
the following:
\begin{corollary}
\label{cor:ind-sort}
Let $d\geq 1$ and 
$a : [N_1]\times\cdots\times[N_d] \rightarrow \reals$ a function.
Let $k<d$ and assume the values $a(\tilde{x})$ are in a descending order
w.r.t.~the first $k$ coordinates, i.e.
\begin{eqnarray*}
a(x_1,\ldots,x_{\ell-1},1,x_{\ell+1},\ldots,x_d)& \geq &  
   a(x_1,\ldots,x_{\ell-1},2,x_{\ell+1},\ldots,x_d) \\
 &      & \vdots\\ 
 & \geq & a(x_1,\ldots,x_{\ell-1},N_{\ell},x_{\ell+1},\ldots,x_d),
\end{eqnarray*}
for each $\ell\in\{1,\ldots,k\}$. If now
$a' : [N_1]\times\cdots\times[N_d] \rightarrow \reals$ is the
function obtained from $a$ by sorting the $N_1\cdots N_kN_{k+2}\cdots N_d$
strings $\left(a(\tilde{x}) : x_{k+1}\in\{1,\ldots,N_{k+1}\}\right)$  
w.r.t.~the $(k+1)$-th coordinate, for each fixed $\pi_{\widehat{k+1}}(\tilde{x})$,
then $a'(\tilde{x})$ are in a descending order w.r.t.~the first
$k+1$ coordinates.
\end{corollary}
By induction we then have by Corollary~\ref{cor:ind-sort}
the following.
\begin{corollary}
\label{cor:gen-sort}
Let $d\geq 1$ and 
$a : [N_1]\times\cdots\times[N_d] \rightarrow \reals$ a function.
Let $a' : [N_1]\times\cdots\times[N_d] \rightarrow \reals$ be the
function obtain from $a$ by first sorting the $N_1\cdots N_{d-1}$
strings $\left(a(\tilde{x}) : x_d\in\{1,\ldots,N_d\}\right)$ w.r.t.~the $d$-th 
coordinate in a descending order, then w.r.t.~the $(d-1)$-th coordinate etc., 
finally sorting w.r.t.~the first coordinate. 
In this case the values $a'(\tilde{x})$ are in a descending order
w.r.t.~each of the $d$ coordinates.
\end{corollary}

Consider a set $S\subseteq {\nats}^d$ of order $n$. 
As a finite set, there are $N_1,\ldots, N_d\in\nats$ such that
$S\in [N_1]\times\cdots\times[N_d]$ and we have the indicator
function  
$\mathbf{1}_S : [N_1]\times\cdots\times[N_d] \rightarrow \{0,1\}$
of $S$, so $\mathbf{1}_S(\tilde{x}) = 1$ iff $\tilde{x}\in S$.
Note that, by definition, the set $g_i(S)$ is the set
whose indicator function $\mathbf{1}_S'$ 
is obtained by sorting the strings $\mathbf{1}_S(\tilde{x})$ 
w.r.t.~the $i$-th coordinate. Hence, gravity along $i$ corresponds
to sorting the indicator function w.r.t.~the $i$-coordinate.
By Corollary~\ref{cor:gen-sort} we have that
\begin{equation}
\label{eqn:ind-nest}
g_ig_1\cdots g_d  = g_1\cdots g_d,
\end{equation}
for any $i\in \{1,\ldots,d\}$. If $g(S) = S$, then 
by (\ref{eqn:ind-nest}) $g_i(S) = S$ for each $i$, and
hence we have Theorem~\ref{thm:fully-nested}.

Our final objective in this section is to show that 
$E_d(S)$ is at maximum when $S$ is fully nested.
\begin{lemma}
\label{lmm:i-nested}
If $S\subseteq {\nats}^d$ is a finite set,
then $E_d(S)\leq E_d(g_i(S))$.
\end{lemma}
\begin{proof}
The edges of $G[S]\subseteq {\nats}^d$ are either parallel
to the $i$-th axis, or not. Each edge parallel to the $i$-th
axis is mapped by $\pi_{\hat{\imath}}$ to a single point 
$\tilde{x}\in {\nats}^{d-1}$.
Each edge that is not parallel to $i$-th axis is mapped
by $\pi_{\hat{\imath}}$ to an edge 
$\{\tilde{x},\tilde{y}\}\in\binom{{\nats}^{d-1}}{2}$.

For $\tilde{x}\in {\nats}^{d-1}$ there are at most 
$|S\cap \pi_{\hat{\imath}}^{-1}(\tilde{x})|-1$ edges of $G[S]$ parallel
to $i$-th axis that
are mapped to $\tilde{x}$ under
$\pi_{\hat{\imath}}$, and there are precisely 
$|g_i(S)\cap \pi_{\hat{\imath}}^{-1}(\tilde{x})|-1 =
|S\cap \pi_{\hat{\imath}}^{-1}(\tilde{x})|-1$ edges of $G[g_i(S)]$ parallel
to $i$-th axis that are mapped to $\tilde{x}$, since they
form a connected path. 

Each edge $\{\tilde{x},\tilde{y}\}\in\binom{{\nats}^{d-1}}{2}$
yields a matching between points/vertices of $G[S]$ 
in $S\cap \pi_{\hat{\imath}}^{-1}(\tilde{x})$ and 
$S\cap \pi_{\hat{\imath}}^{-1}(\tilde{y})$.
In particular, if $k$ edges in $G[S]$ are mapped to $\{\tilde{x},\tilde{y}\}$, 
then both $S\cap \pi_{\hat{\imath}}^{-1}(\tilde{x})$
and $S\cap \pi_{\hat{\imath}}^{-1}(\tilde{y})$ have cardinality of $k$ or greater.
Therefore the number of edges in $G[g_i(S)]$ is given by
\[
\min(|g_i(S)\cap \pi_{\hat{\imath}}^{-1}(\tilde{x})|,
     |g_i(S)\cap \pi_{\hat{\imath}}^{-1}(\tilde{y})|) = 
\min(|S\cap \pi_{\hat{\imath}}^{-1}(\tilde{x})|,
     |S\cap \pi_{\hat{\imath}}^{-1}(\tilde{y})|) \geq k,
\]
and so $G[g_i(S)]$ has at least $k$ edges mapped to
$\{\tilde{x},\tilde{y}\}$. 
\end{proof}
\begin{corollary}
\label{cor:max-fully-nested}
Among all finite sets $S\subseteq {\nats}^d$ then
$E_d(S)$ is at maximum when $S$ is fully nested.
\end{corollary}

\section{Tight asymptotic bounds}
\label{sec:asymp}

The objective in this section is to derive an
asymptotically tight upper bound for $E_d(n)$, which
by Corollary~\ref{cor:max-fully-nested}, equals 
$E_d(S)$ for some finite fully nested set $S\subseteq {\nats}^d$ 
of order $n$.

Let $S\subseteq {\nats}^d$ be a fully nested set of order $n$. 
The edges of $G[S]$ are partitioned into $i$ parts, 
the $i$-th part consisting of all edges parallel to the $i$-th axis. 
As noted in the proof of Lemma~\ref{lmm:i-nested}, each
point $\tilde{x}\in \pi_{\hat{\imath}}(S)$ corresponds
to a connected path of $G[S]$ since $S$ is fully nested. Also, since there
are $n$ vertices of $G[S]$, and 
exactly $n_{\hat{\imath}} =n_{\hat{\imath}}(S) := |\pi_{\hat{\imath}}(S)|$ 
disjoint paths of $G[S]$ 
parallel to $i$-axis, the number of edges parallel to $i$-th axis
is $n - n_{\hat{\imath}}$. From this we have the following exact 
count on the number of edges of $G[S]$.
\begin{observation}
\label{obs:exact-edge-count}
If $S\subseteq {\nats}^d$ is fully nested set with $|S| = n$,
then $E_d(S) = dn - (n_{\hat{1}} + \cdots + n_{\hat{d}})$.
\end{observation}
From the above Observation~\ref{obs:exact-edge-count} we see
that if we can compute the exact minimum value of $\sum_i n_{\hat{\imath}}$
for all fully nested sets   $S\subseteq {\nats}^d$ of order $n$,
then we can determine $E_d(n)$ by subtracting that minimum value from $dn$.

By a theorem of Bollob\'{a}s and 
Thomason~\cite[Thm 2, p.418]{Bollobas-Thomason} we have
\begin{equation}
\label{eqn:BoxThm}
n^{d-1}\leq\prod_{i=1}^dn_{\hat{\imath}}.
\end{equation}
Note that equality holds in (\ref{eqn:BoxThm}) 
for any set $S$ of the form $S = S_1\times\cdots\times S_d$.
By Observation~\ref{obs:exact-edge-count}, the inequality of
arithmetic and  geometric mean, and (\ref{eqn:BoxThm}) we obtain
\begin{equation}
\label{eqn:S-ass-tight}
E_d(S) = dn - \left(\sum_{i=1}^d n_{\hat{\imath}}\right)
      \leq d\left(n - \sqrt[d]{\prod_{i=1}^dn_{\hat{\imath}}}\right)
      \leq d(n - \sqrt[d]{n^{d-1}})
         = dn(1 - n^{-1/d})
\end{equation}
and therefore for each $n\in\nats$ we have 
\begin{equation}
\label{eqn:n-ass-tight}
E_d(n) \leq dn(1 - n^{-1/d}).
\end{equation}
Since equality holds in the inequality of arithmetic and geometric
mean iff all the parameters are equal, we have that
equality holds in (\ref{eqn:S-ass-tight}) for a fully nested set $S$
iff $S = [m]^d$ is a $d$-dimensional hypercube (or a {\em $d$-cube} for
short) with $m^d$ vertices.
Hence for $n = m^d$ we have equality in (\ref{eqn:n-ass-tight}).
For each fixed $d$ both the functions $E_d(n)$ and the upper bound
on the right of (\ref{eqn:n-ass-tight}) are clearly increasing functions
of $n$. Also, since $E_d(n)$ is always an integer we have
the following.
\begin{proposition}
\label{prp:n-ass-tight}
For all $d,n\in\nats$ we have
\begin{equation}
\label{eqn:n-int-tight}
E_d(n) \leq \lfloor dn(1 - n^{-1/d})\rfloor.
\end{equation}
This bound is asymptotically tight as $n\rightarrow\infty$
and equality holds for all $d$-th powers $n=m^d$.
\end{proposition}
{\sc Remark:} Note that for $d\in\{1,2\}$ then we have
equality in (\ref{eqn:n-int-tight}) as shown in~\cite{Harary-Harborth}.

\section{A recursive inequality}
\label{sec:rec-ineq}

In this section we will derive a general recursive upper bound of $E_d(n)$,
that is tight in the sense that it can be realized in some specific
cases. By Corollary~\ref{cor:max-fully-nested} it suffices to consider
fully nested $S\subseteq {\nats}^d$ of order $n$, and by 
Proposition~\ref{prp:n-ass-tight} we may, when necessary, 
assume that $n$ is not a $d$-th power of an integer.

{\sc Convention:} Just like a $d$-cube had $2d$ sides,
a fully nested set $S\subseteq {\nats}^d$ will, in our context,
have $2d$ sides as well, namely $S\cap H$ where $H : x_i = k$,
$i\in [d]$ and $k\in \{1,|\pi_i(S)|\}$
is one of the $2d$ supporting hyperplanes of $S$.

For any fully nested $S\subseteq {\nats}^d$ not contained in a
hyperplane, and any hyperplane 
$H_k : x_d = k$, where $2\leq k\leq |\pi_d(S)|$,
we obtain a partition or {\em cut} by
\begin{eqnarray*}
S_1  & = & \{\tilde{x}\in S : \pi_d(\tilde{x}) < k\}, \\
S_2  & = & \{\tilde{x}\in S : \pi_d(\tilde{x}) \geq k\}.
\end{eqnarray*}
Assume that $|S|=n$ and that $S$ is optimal, so
$E_d(S) = E_d(n)$. Since $S$ is fully nested the number
of edges parallel to the $x_d$ axis that cut through the
hyperplane $H_k$, in the sense that one endvertex is in $S_2$
and the other is in $S_1$, is given by $n_{\hat{d}}(S_2) = |\pi_{\hat{d}}(S_2)|$.
From this we see that for our set $S$ we then have
\begin{equation}
\label{eqn:rec-exact1}
E_d(n) = E_d(S) = E_d(S_1) + E_d(S_2) + n_{\hat{d}}(S_2).
\end{equation}
{\sc Note:} There is no significance to the last coordinate $x_d$ here.
This can also be obtained by any cut perpendicular to any of the $d$
coordinate axes.

\vspace{3 mm}

If $h = |\pi_d(S)|$, then we have a partition 
$S_2 = S_{2;k}\cup\cdots\cup S_{2;h}$ where each  
$S_{2;i} = \{\tilde{x}\in S : \pi_d(\tilde{x}) = i\}$.
Since $S$ is fully nested we have 
\[
\pi_{\hat{d}}(S_{2;k})\supseteq\cdots\supseteq \pi_{\hat{d}}(S_{2;h})
\]
and $n_{\hat{d}}(S_2) = |\pi_{\hat{d}}(S_{2;k})|= n_{\hat{d}}(S_{2;k})$ 
and therefore  
\[
E_d(S_2) + n_{\hat{d}}(S_2) 
= \sum_{i=k}^h\left(n_{\hat{d}}(S_{2;i}) + E_{d-1}(\pi_{\hat{d}}(S_{2;i})\right).
\]
Since 
\[
\sum_{i=k}^hn_{\hat{d}}(S_{2;i}) = \sum_{i=k}^h |S_{2;i}| = |S_2|,
\]
we obtain   
\begin{equation}
\label{eqn:S2-part}
E_d(S_2) + n_{\hat{d}}(S_2) = |S_2| + \sum_{i=k}^hE_{d-1}(\pi_{\hat{d}}(S_{2;i})
\end{equation}
By (\ref{eqn:defE(n)}), the definition of $E_{d-1}$ as a
function $\nats\rightarrow\nats$, we clearly
have $\sum_{i=k}^hE_{d-1}(\pi_{\hat{d}}(S_{2;i}))\leq E_{d-1}(|S_2|)$
and hence from (\ref{eqn:S2-part}) we then get 
\[
E_d(S_2) + n_{\hat{d}}(S_2) \leq |S_2| + E_{d-1}(|S_2|)
\]
and hence from (\ref{eqn:rec-exact1}) we get the inequality
\begin{equation}
\label{eqn:gen-rec-ineq}
E_d(n) = E_d(S) \leq E_d(|S_1|) + E_{d-1}(|S_2|) + |S_2|. 
\end{equation}
We summarize in the following.
\begin{lemma}
\label{lmm:rec-ineq}
Let $n\in\nats$ and $S\subseteq {\nats}^d$
fully nested and optimal with $|S|=n$.
Then for any cut that partitions $S$ into two proper sets 
$S_1$ and $S_2$ of order $n_1$ and $n_2$ respectively, 
and so $n_1+n_2 = n$, we have 
\begin{equation}
\label{eqn:rec-ineq}
E_d(n) \leq E_d(n_1) + E_{d-1}(n_2) + n_2.
\end{equation}
\end{lemma}
Note that (\ref{eqn:rec-ineq}) does not hold for 
any partition $n = n_1+n_2$; only for the mentioned
particular partitions. The main thing to notice 
in Lemma~\ref{lmm:rec-ineq} is that there {\em exists}
a proper partition of $n$ that yields the desired
inequality.

Let $S\subseteq {\nats}^d$ be a fully nested and optimal with $|S|=n$.
Of particular interest is the special cut when $k=h = |\pi_d(S)|$,
so $S_2$ is a $(d-1)$-dimensional side of $S$.
In this case $S_2 = S_{2;h}$ and
so by (\ref{eqn:rec-exact1}) and (\ref{eqn:S2-part}) we obtain
\begin{equation}
\label{eqn:rec-exact2}
E_d(n) = E_d(S) = E_d(S_1) + E_{d-1}(\pi_{\hat{d}}(S_2)) + |S_2|.
\end{equation}
Note that if in addition $S_1$ and $S_2$ are also optimal,
then (\ref{eqn:rec-exact2}) will yield an equality in 
(\ref{eqn:rec-ineq}).
We will see that such an equality can be
obtained in (\ref{eqn:rec-ineq}).
With this in mind, it is our next objective to show that for each
$n\in\nats$ there is always a fully nested and optimal set $S$ of order $n$ 
and a cut with a partition $S = S_1\cup S_2$ where $S_2$ is a side of $S$
such that equality holds in (\ref{eqn:rec-ineq}). To do that we
need results in the next section.

\section{Pseudo cubes, pseudo cubics and their properties}
\label{sec:cubics}

In this section we define some specific representations for 
integers, their corresponding
sets in ${\nats}^d$, and prove some properties that will demonstrate 
that we can always assume that an optimal $S\subseteq {\nats}^d$ is 
one of these corresponding sets.
\begin{definition}
\label{def:pseudo-cube}
For $d\in\nats$, call a number $n\in\nats$ a {\em pseudo $d$-cubic} 
if $n = (m+1)^{\ell}m^{d-\ell} := [m,\ell]^d$ for some $m\in\nats$ and 
$\ell\in\{0,1,\ldots,d-1\}$. 
A {\em pseudo cubic} is then a pseudo $d$-cubic for some $d$.

Any pseudo $d$-cubic $n$ yields a corresponding 
{\em pseudo $d$-cube} $\llbracket m,\ell\rrbracket^d 
:= [m+1]^{\ell}\times[m]^{d-\ell}\subseteq {\nats}^d$.
\end{definition}
{\sc Remarks:} (i) Although the word ``cubic'' is an adjective,
we will use it both as such and also as a noun.
(ii) We do reserve the right to interpret $[m,\ell]^d$ when $\ell = d$
by the defining algebraic expression in Definition~\ref{def:pseudo-cube},
so $[m,d]^d = [m+1,0]^d$. 
(iii) A pseudo $d$-cube $\llbracket m,\ell\rrbracket^d$ has $2d$ sides; 
$2\ell$ of which are copies of 
$\llbracket m,\ell-1\rrbracket^{d-1}$, and $2(d-\ell)$ of which
are copies of $\llbracket m,\ell\rrbracket^{d-1}$, both types of sides are
pseudo $(d-1)$-cubes.

\vspace{3 mm}

With the above remark in mind, then clearly for $d$ fixed,
every $n\in\nats$ is between two
pseudo $d$-cubics: $[m,\ell]^d \leq n < [m,\ell+1]^d$ for some $m$ and 
$\ell \in\{0,1,\ldots,d-1\}$. Since 
$(m+1)^{\ell+1}m^{d-\ell-1} - (m+1)^{\ell}m^{d-\ell} = (m+1)^{\ell}m^{d-\ell-1}$,   
then the difference between two consecutive pseudo $d$-cubics
is a pseudo $(d-1)$-cubic.

Recall the lexicographical ordering on ${\ints}^d$:
\[
\tilde{x} < \tilde{y} \Leftrightarrow 
x_i = y_i \mbox{ for } 1\leq i\leq j-1 \mbox{ and } x_j < y_j.
\]
The lexicographical ordering is a total/linear ordering 
of the elements of ${\ints}^d$. 

With the above convention we have similarly to the 
Pascal's Rule for binomial coefficient the following. 
\begin{claim}
\label{clm:pc-diff}
For a fixed $m$ and $\ell\in \{0,1,\ldots,d-1\}$ 
we have for pseudo cubics that 
\[
[m,\ell+1]^d = [m,\ell]^d + [m,\ell]^{d-1}.
\]
Also, for a fixed $d$ we have $[m,\ell]^d\leq [m',\ell']^d$
iff $(m,\ell)\leq (m',\ell')$ lexicographically.
\end{claim}
Note that although the partition in Claim~\ref{clm:pc-diff}
could be defined for all integer values of $\ell$, including negative $\ell$, it
is an integer partition only for $\ell\in\{0,1,\ldots,d-1\}$.

In a similar fashion to the unique binomial representation of an 
integer~\cite[p.~55]{Stanley-comb-comm} and~\cite[Lemma 7.1]{Hibi},
Claim~\ref{clm:pc-diff} yields the following. 
\begin{proposition}
\label{prp:!-pc-rep}
Every $n\in\nats$ has a unique 
{\em pseudo $d$-cubic representation ($d$-PCR)} as
\[
n = [m_d,{\ell}_d]^d + [m_{d-1},{\ell}_{d-1}]^{d-1} + \cdots 
+ [m_{c},{\ell}_{c}]^{c} 
\]
where $c, m_{c} \geq 1$ and 
$(m_d,{\ell}_d) > (m_{d-1},{\ell}_{d-1}) > \cdots > (m_{c},{\ell}_{c})$ 
lexicographically.
\end{proposition}
\begin{proof}
We proceed in a greedy fashion; for a given $n$ we choose 
the unique $(m_d,{\ell}_d)$ such that
$[m_d,{\ell}_d]^d$ is the largest pseudo $d$-cubic less than or equal to $n$. 
We continue by letting $[m_{d-1},{\ell}_{d-1}]^{d-1}$ be the largest 
pseudo $(d-1)$-cube less than or equal to $n-[m_d,{\ell}_d]^d$.
By Claim~\ref{clm:pc-diff} we have then have 
$(m_d,{\ell}_d) > (m_{d-1},{\ell}_{d-1})$. The rest follows by induction 
on $n$.
\end{proof}
When either $d$ is fixed or irrelevant, we will just write PCR 
for $d$-PCR of a natural number $n$.

\vspace{3 mm}

{\sc Remarks:} (i) Note that for the $d$-PCR of $n$ we have 
$m_d = \lfloor \sqrt[d]{n}\rfloor$. (ii) Also, since 
$[m_d,{\ell}_d]^d \leq n < [m_d,{\ell}_d+1]^d$ then 
\[
{\ell}_d = \left\lfloor \frac{\log(n/m_d^d)}{\log(1+1/m_d)}\right\rfloor.
\]
This can be used to obtain a quick recursive method to obtain the 
$d$-PCR of $n$ as indicated in Observation~\ref{ops:PCR-recursive}
here below.
(ii) By letting $m_1 = m_2 = \cdots = m_{c - 1} = 0$ and
${\ell}_1 = {\ell}_2 = \cdots = {\ell}_{c - 1} = 0$, and noting that 
$[0,0]^d = 0$, we can, when needed, assume each $d$-PCR to have 
exactly $d$ terms. 
\begin{observation}
\label{obs:PCR-compare}
If $n,n'\in\nats$ have $d$-PCR given by
\begin{eqnarray*}
n & = & [m_d,{\ell}_d]^d + [m_{d-1},{\ell}_{d-1}]^{d-1} + 
        \cdots + [m_1,{\ell}_1]^1, \\
n' & = & [m'_d,{\ell}'_d]^d + [m'_{d-1},{\ell}'_{d-1}]^{d-1} + 
\cdots + [m'_1,{\ell}'_1]^1, \\
\end{eqnarray*}
then $n\leq n'$ iff $(m_d,{\ell}_d,m_{d-1},{\ell}_{d-1},\ldots,m_1,{\ell}_1)
\leq (m'_d,{\ell}'_d,m'_{d-1},{\ell}'_{d-1},\ldots,m'_1,{\ell}'_1)$
lexicographically.
\end{observation}

The following observation is convenient when working recursively.
\begin{observation}
\label{ops:PCR-recursive}
If $n\in\nats$ has a $d$-PCR given by
\[
n = [m_d,{\ell}_d]^d + [m_{d-1},{\ell}_{d-1}]^{d-1} + \cdots 
+ [m_{c},{\ell}_{c}]^{c}, 
\]
then the $(d-1)$-PCR of $n-[m_d,{\ell}_d]^d$ is given by
\[
n - [m_d,{\ell}_d]^d = [m_{d-1},{\ell}_{d-1}]^{d-1} + \cdots 
+ [m_{c},{\ell}_{c}]^{c}. 
\]
\end{observation}
The PCR of an integer gives rise to a special fully nested configuration.
In order to describe this we need the following definition.
\begin{definition}
\label{def:lift}
For a set $J\subseteq[d]$ with $|J|=j$
and a fixed point $\tilde{a}\in{\nats}^j$ we have a 
lifting map 
\[
\lambda_{J;\tilde{a}} : {\ints}^{d-j} \rightarrow {\ints}^d,
\]
such that 
$\pi_J\circ\lambda_{J;\tilde{a}} : {\ints}^{d-j} \rightarrow {\ints}^j$
is the constant map taking each element to $\tilde{a}$,
and $\pi_{[d]\setminus J}\circ\lambda_{J;\tilde{a}} : 
{\ints}^{d-j} \rightarrow {\ints}^{d-j}$ is the identity map.
\end{definition}
If the set $J$ is given explicitly $J = \{h_1,\ldots,h_j\}$,
then we usually write $\lambda_{J;\tilde{a}}$ as 
$\lambda_{h_1,\ldots,h_j;\tilde{a}}$.

From the $d$-PCR $n = [m_d,{\ell}_d]^d + \cdots + [m_{c},{\ell}_{c}]^{c}$
where $c, m_{c} \geq 1$,
we obtain the set $\llbracket n\rrbracket^d\subseteq {\nats}^d$ 
of order $n$ recursively by setting 
$\llbracket 0\rrbracket^i := \emptyset$ for each $i$ and 
\begin{equation}
\label{eqn:cubicle-rec}
\llbracket n\rrbracket^d := 
\llbracket m_d,{\ell}_d\rrbracket^d \cup
   \lambda_{{\ell}_d+1;m_d+1}(\llbracket n - [m_d,{\ell}_d]^d\rrbracket^{d-1}).
\end{equation}
We list some properties of these sets 
$\llbracket n\rrbracket^d\subseteq {\nats}^d$ that are immediate.
From the 
$d$-PCR $n = [m_d,{\ell}_d]^d + \cdots + [m_{c},{\ell}_{c}]^{c}$
we have the following.
\begin{proposition}
\label{prp:cubicle}
For $n\in\nats$ we have 
\begin{enumerate}
  \item $|\llbracket n\rrbracket^d| = n$.
  \item $\llbracket n\rrbracket^d\subseteq \llbracket n'\rrbracket^d$ 
iff $n\leq n'$.
  \item $\llbracket n\rrbracket^d$ is fully nested.
\end{enumerate}
\end{proposition}
\begin{proof}
Since the defining union in the recursion (\ref{eqn:cubicle-rec}) is disjoint,
the first assertion follows by induction on $n$.

The second assertion follows from Proposition~\ref{prp:!-pc-rep}.

For the third assertion we see that since 
$\llbracket m_d,{\ell}_d\rrbracket^d$ is fully nested,
then for each $i\neq {\ell}_m+1$, $\llbracket n\rrbracket^d$ is $i$-nested
iff 
$\lambda_{{\ell}_d+1;m_d+1}(\llbracket n - [m_d,{\ell}_d]^d\rrbracket^{d-1})$
is $i$-nested, which by induction on $n$ is $i$-nested. Since
$\llbracket m_d,{\ell}_d\rrbracket^{d-1}$ contains 
$\llbracket n - [m_d,{\ell}_d]^d\rrbracket^{d-1}$ we also have 
$({\ell}_m+1)$-nestedness.
\end{proof}
\begin{definition}
\label{def:cubicle}
For $n\in\nats$, a set $\llbracket n\rrbracket^d\subseteq {\nats}^d$
as given above, is called a {\em $d$-cubicle}, or simply a 
{\em cubicle} if $d$ is irrelevant.
\end{definition}
As with cubics, we can also talk about a {\em side} of a cubicle as 
the intersection
of the cubicle with one of its supporting hyperplanes, most notably the 
planes $x_i = 1$ for various $i$. These $d$ sides of 
$\llbracket n\rrbracket^d$ are given by 
$\pi_{\hat{\imath}}(\llbracket n\rrbracket^d)$ for each $i$.
By the recursive definition (\ref{eqn:cubicle-rec}) we also have 
the following.
\begin{claim}
\label{clm:side}
For each $n\in\nats$ we have
that $\pi_{\hat{\imath}}(\llbracket n\rrbracket^d)$ is a $(d-1)$-cubicle.
Hence, each projection of $\llbracket n\rrbracket^d$ is also a cubicle.
\end{claim}
Our next lemma will be useful in the next section.
\begin{lemma}
\label{lmm:side}
For $n\in\nats$ with $d$-PCR 
$n = [m_d,{\ell}_d]^d + \cdots + [m_{c},{\ell}_{c}]^{c}$, then 
$\llbracket n\rrbracket^d\subseteq {\nats}^d$ always has a side $S$
with $|S| \geq \frac{n}{m_d+1}$. If further 
${\ell}_d<d-1$ then every side 
$S = \pi_{\hat{\imath}}(\llbracket n\rrbracket^d)$ where 
$i\in\{{\ell}_d+2,\ldots,d\}$ satisfies $|S|\geq \frac{n}{m_d}$.
\end{lemma}
\begin{proof}
Since $\llbracket n\rrbracket^d\subseteq \llbracket m_d+1,0\rrbracket^d$
and $\llbracket n\rrbracket^d$ is fully nested we clearly 
have $|\pi_{\hat{\imath}}(\llbracket n\rrbracket^d)|\geq \frac{n}{m_d+1}$.

If further ${\ell}_d<d-1$, then for each 
$i\in\{{\ell}_d+2,\ldots,d\}\neq\emptyset$ we have 
$|\pi_i(\llbracket n\rrbracket^d)| = m_d$ and since
$\llbracket n\rrbracket^d$ is fully nested we therefore have
$|\pi_{\hat{\imath}}(\llbracket n\rrbracket^d)|\geq \frac{n}{m_d}$.
\end{proof}
Our ultimate goal is to show that the $d$-cubicles are sets
achieving the most edges among induced graphs in ${\nats}^d$.
\begin{definition}
\label{def:F}
For $n\in\nats$ let $F_d(n)$ be the number of edges 
that the $d$-cubicle $\llbracket n\rrbracket^d$ induces in ${\nats}^d$,
that is $F_d(n) := E_d(\llbracket n\rrbracket^d)$.
\end{definition}
Our goal is therefore to show that $E_d(n) = F_d(n)$, although 
$\llbracket n\rrbracket^d$ is by no means the unique fully nested
configuration in ${\nats}^d$ yielding the maximum number $E_d(n)$ 
of edges. The rest of this current section and the following next section
will be devoted to obtain this goal.

In the same fashion as we derived
(\ref{eqn:rec-exact2}) we get by Claim~\ref{clm:pc-diff}
the following recursion for each $\ell\in\{0,1,\ldots,d\}$.
\begin{equation}
\label{eqn:F(pc)-rec}
F_d([m,\ell+1]^d) = F_d([m,\ell]^d) + F_{d-1}([m,\ell]^{d-1}) + [m,\ell]^{d-1}.
\end{equation}
Likewise, for $n\in\nats$ with $d$-PCR 
$n = [m_d,{\ell}_d]^d + [m_{d-1},{\ell}_{d-1}]^{d-1} + \cdots 
+ [m_{c},{\ell}_{c}]^{c}$,
we get by the recursive definition of $\llbracket n\rrbracket^d$
and Definition~\ref{def:F} that
\begin{equation}
\label{eqn:F(n)-rec}
F_d(n) = F_d([m_d,\ell_d]^d) + F_{d-1}(n-[m_d,\ell_d]^d) + n-[m_d,\ell_d]^d.
\end{equation}
By Observation~\ref{ops:PCR-recursive} and (\ref{eqn:F(n)-rec}) we obtain
recursively an expression for $F_d(n)$, namely
\begin{equation}
\label{eqn:F(n)-sum}
F_d(n) = 
\sum_{i=c}^d \left(F_i([m_i,{\ell}_i]^i) + (d-i)[m_i,{\ell}_i]^i]\right),
\end{equation}
where for each pseudo cube we again obtain recursively by 
(\ref{eqn:F(pc)-rec}) that
\begin{equation}
\label{eqn:F(pc)-exactl}
F_d([m,\ell]^d) = d[m,\ell]^d - \ell[m,\ell-1]^{d-1} - (d-\ell)[m,\ell]^{d-1},
\end{equation}
for $\ell\in\{0,1,\ldots,d-1\}$. Note that for $\ell\in\{0,d-1\}$
then (\ref{eqn:F(pc)-exactl}) yields a valid formula (one with $m$ 
and the other with $m+1$), which by itself can be verified by induction
using (\ref{eqn:F(pc)-rec}) as well. Hence, (\ref{eqn:F(pc)-exactl}) yields
an explicit formula for $F_d(n)$ for every pseudo $d$-cubic $n$.

{\sc Remarks:} (i) The formula (\ref{eqn:F(pc)-exactl})
for $F_d([m,\ell]^d)$ can also be obtained from 
Observation~\ref{obs:exact-edge-count}.
(ii) Note that (\ref{eqn:F(pc)-exactl}) for $\ell\in\{0,d-1\}$ matches the
upper bound given in (\ref{eqn:n-int-tight}) for $n = [m,0]^d$, 
which shows that $F_d(n) = E_d(n)$ for every $d$-power of an integer $n$,
something already stated clearly in Proposition~\ref{prp:n-ass-tight}.

\vspace{3 mm}

From (\ref{eqn:F(n)-sum}) and (\ref{eqn:F(pc)-exactl}) 
we then get the following explicit formula for $F_d(n)$.
\begin{observation}
\label{obs:F-formula}
For $n\in\nats$ with the above $d$-PCR we then have 
$F_d(n) = dn - \delta_d(n)$ where the discrepancy is given by 
\[
\delta_d(n) = \sum_{i=c}^d
\left(\ell_i[m_i,\ell_i-1]^{i-1} + (i-\ell_i)[m_i,\ell_i]^{i-1}\right).
\]
\end{observation}

We now prove some important properties of the function $F_d$. 
In order to do that we need to introduce some notation.

\paragraph{Notation:} Let $n\in\nats$ with $d$-PCR
$n = [m_d,{\ell}_d]^d + [m_{d-1},{\ell}_{d-1}]^{d-1} + \cdots + 
[m_{c},{\ell}_{c}]^{c}$. 
(i) Let $[n]^d_{-} := [m_d,{\ell}_d]^d$ 
denote the largest pseudo cubic $\leq n$, 
so $n = [n]^d_{-} + n'$ where 
$n' = [m_{d-1},{\ell}_{d-1}]^{d-1} + \cdots 
+ [m_{c},{\ell}_{c}]^{c} < [m_d,{\ell}_d]^{d-1}$,
and 
(ii) let $[n]^d_{+}$ 
denote the smallest pseudo cubic $>n$, so for each 
${\ell}_d\in\{0,1,\ldots,d-1\}$
\[
[n]^d_{+} := [m_d,{\ell}_d + 1]^d.
\]
Note that $[n]^d_{-}$ and $[n]^d_{+}$ are consecutive pseudo cubics
and $[n]^d_{-}\leq n < [n]^d_{+}$. 
(iii) Let
\[
[n]^{d-1}_{\Delta} := [n]^d_{+} - [n]^d_{-},
\]
so $[n]^{d-1}_{\Delta} = [m_d,{\ell}_d]^{d-1}$ in terms of the 
$d$-PCR of $n$ above. Since $m_d = \lfloor\sqrt[d]{n}\rfloor$
in terms of $n$ and $d$ alone, we obtain by partitioning 
$\llbracket n\rrbracket^d$ into ``slices''
of height one and order $[n]^{d-1}_{\Delta}$ 
along the $({\ell}_d + 1)$-th coordinate that
\begin{equation}
\label{eqn:F(pc)-sliced}
F_d([n]^d_{-}) =  \lfloor\sqrt[d]{n}\rfloor 
F_{d-1}([n]^{d-1}_{\Delta}) + 
(\lfloor\sqrt[d]{n}\rfloor -1)[n]^{d-1}_{\Delta},
\end{equation}
and hence for each $i\in \{0,1,\ldots, \lfloor\sqrt[d]{n}\rfloor\}$
we then obtain by (\ref{eqn:F(n)-rec}) and (\ref{eqn:F(pc)-sliced}) 
that
\begin{equation}
\label{eqn:F(n)-i}
F_d(n) = F_{d-1}(n') + n' + i(F_{d-1}([n]^{d-1}_{\Delta}) + [n]^{d-1}_{\Delta})
+ (\lfloor\sqrt[d]{n}\rfloor -i)F_{d-1}([n]^{d-1}_{\Delta})
+ (\lfloor\sqrt[d]{n}\rfloor -i-1)[n]^{d-1}_{\Delta},
\end{equation}
where $n' = n - [n]^d_{-} = n - [m_d,{\ell}_d]^d$. 
(iv) Needless to say we can recursively define
\begin{eqnarray*}
[n]^d_{1\/+}    & := & [n]^d_{+}, \\
{[n]}^d_{(i+1)\/+} & := & [[n]^d_{i\/+}]^d_{+}, \mbox{ for }i\geq 1
\end{eqnarray*}
thereby obtaining a strictly increasing sequence of
consecutive pseudo $d$-cubics
\[
[n]^d_{+} < [n]^d_{2\/+} < \cdots < [n]^d_{i\/+} < \cdots,
\]
the unique such sequence that contains every pseudo $d$-cubic strictly
larger than $n$. Also, this can be done in the negative
direction as well to obtain 
$[n]^d_{-} > [n]^d_{2\/-} > \cdots > [n]^d_{i\/-} > \cdots$, the unique
sequence containing $d$-cubics less than or equal to $n$. However, here
the recursion is slightly different, as $[n]_{-}^d$ is the largest 
pseudo cubic $\leq n$ as suppose to $< n$.
\begin{eqnarray*}
[n]^d_{1\/-}    & := & [n]^d_{-}, \\
{[n]}^d_{(i+1)\/-} & := & [[n]^d_{i\/-} - 1]^d_{-}, \mbox{ for }i\geq 1.
\end{eqnarray*}

Let $d,n\in\nats$ be fixed and consider the following statements. 
\begin{quote}
$\mathbf{P}(d,n)$ : $F_d(n_1) + F_d(n_2) \leq F_d([m,\ell]^d) + F_d(n')$,
whenever $n_1+n_2 = [m,\ell]^d + n' = n$ and $n_1,n_2\leq [m,\ell]^d$.
\end{quote}
\begin{quote}
$\mathbf{P}'(d,n)$ : $F_d(n_1) + F_d(n_2) \leq F_d([n_1]^d_{+}) + F_d(n')$,
whenever $n_1+n_2 = [n_1]^d_{+} + n' = n$ and $n_2\leq n_1$.
\end{quote}
For a fixed $d,n\in\nats$ we clearly have the implication
$\mathbf{P}(d,n)\Rightarrow \mathbf{P}'(d,n)$.
We now briefly argue the reverse implication
$\mathbf{P}'(d,n)\Rightarrow \mathbf{P}(d,n)$.

Assume $\mathbf{P}'(d,n)$ and let $n_1+n_2 = [m,\ell]^d + n' = n$ where 
$n_1,n_2\leq [m,\ell]^d$. With the notation above, there is an finite sequence
\[
[n_1]^d_{+} < \cdots < [n_1]^d_{j\/+} = [m,\ell]^d,
\]
and with repeated use of $\mathbf{P}'(d,n)$ we obtain
\[
F_d([n_1]^d_{i\/+}) + F_d(n^{(i)})\leq F_d([n_1]^d_{(i+1)\/+}) + F_d(n^{(i+1)})
\]
for each $i$, where $[n_1]^d_{i\/+}+n^{(i)} = [n_1]^d_{(i+1)\/+} + n^{(i+1)} = n$,
which yields $F_d(n_1) + F_d(n_2) \leq F_d([m,\ell]^d) + F_d(n')$.
This proves $\mathbf{P}'(d,n)\Rightarrow \mathbf{P}(d,n)$,
and so we have $\mathbf{P}'(d,n)\Leftrightarrow \mathbf{P}(d,n)$.

For $d,n\in\nats$ let our goal be phrased as the following statement.
\begin{quote}
$\mathbf{EF}(d,n)$ : $E_d(n) = F_d(n)$.
\end{quote}
To prove that $\mathbf{EF}(d,n)$ is valid for every $d,n\in\nats$,
we first note that both $\mathbf{EF}(d,n)$ and $\mathbf{P}(d,n)$
are trivially true whenever either $d=1$ or $n=1$. We then
proceed to show that for any $N\in\nats$ 
\begin{equation}
\label{eqn:main-impl}
\mathbf{P}(d,n)\wedge\mathbf{EF}(d,n)\mbox{ for }d+n<N
\Rightarrow
\mathbf{P}'(d,n)\wedge\mathbf{EF}(d,n)\mbox{ for }d+n = N.
\end{equation}
We conclude this section by proving the following implication
\begin{equation}
\label{eqn:1st-impl}
\mathbf{P}(d,n)\wedge\mathbf{EF}(d,n)\mbox{ for }d+n<N
\Rightarrow
\mathbf{P}'(d,n)\mbox{ for }d+n = N.
\end{equation}
The remainder of (\ref{eqn:main-impl}) will be proved
in the following section.

To prove (\ref{eqn:1st-impl}) let $d,n,n_1,n_2,N\in\nats$ and $n'\geq 0$
be such that $d+n=N$, $n_2\leq n_1$, and $n_1+n_2 = [n_1]^d_{+} + n' = n$. 
Write $n_1 = [n_1]^d_{-} + n_1'$, $n_2 = [n_2]^d_{-} + n_2'$, and 
$[n_1]^d_{+} =  [n_1]^d_{-} +  [n_1]^{d-1}_{\Delta}$. 
By (\ref{eqn:F(pc)-rec}) and (\ref{eqn:F(n)-rec}) we get 
\begin{eqnarray*}
F_d([n_1]^d_{+}) & = & F_d([n_1]^d_{-}) + F_d([n_1]^{d-1}_{\Delta}) 
+ [n_1]^{d-1}_{\Delta}, \\
F_d(n_1) & = & F_d([n_1]^d_{-}) + F_{d-1}(n_1') + n_1'.
\end{eqnarray*}
Let $i\in \{0,1,\ldots, \lfloor \sqrt[d]{n_2}\rfloor\}$
be such that 
\[
(i-1)[n_2]^{d-1}_{\Delta} + n_2' + n_1' < [n_1]^{d-1}_{\Delta} \leq
i[n_2]^{d-1}_{\Delta} + n_2' + n_1'.
\]
Such an $i$ exists: since $n_2\leq n_1$ and $n_1+n_2 = [n_1]^d_{+}+n'$
where $n'\geq 0$, we get by subtracting $[n_1]_{-}^d$ from each side
and rewriting $n_2$, that 
\begin{eqnarray*}
n_2' + \lfloor\sqrt[d]{n_2}\rfloor[n_2]^{d-1}_{\Delta} + n_1' 
 & = & n_2 + n_1' \\ 
 & = & n_2 + (n_1 - [n_1]_{-}^d) \\
 & = & n' + ([n_1]_{+}^d - [n_1]_{-}^d) \\
 & = & n' + [n_1]^{d-1}_{\Delta}\\ 
 &\geq & [n_1]^{d-1}_{\Delta}. 
\end{eqnarray*}
Hence, we have $i[n_2]^{d-1}_{\Delta} + n_2' + n_1' = [n_1]^{d-1}_{\Delta}+n''$
where $0\leq n''<[n_2]^{d-1}_{\Delta}$. Note that if $i=0$
then $n_1'+n_2'\geq [n_1]^{d-1}_{\Delta}\geq [n_2]^{d-1}_{\Delta}$
and so $(i-1)[n_2]^{d-1}_{\Delta} + n_2' + n_1'\geq 0$.
By (\ref{eqn:F(n)-i}) we then
get 
\[
F_d(n_2) =  
  i(F_{d-1}([n_2]^{d-1}_{\Delta}) + [n_2]^{d-1}_{\Delta}) + 
  (\lfloor\sqrt[d]{n_2}\rfloor -i)F_{d-1}([n_2]^{d-1}_{\Delta}) + 
  (\lfloor\sqrt[d]{n_2}\rfloor -i-1)[n_2]^{d-1}_{\Delta} + 
  F_{d-1}(n_2') + n_2'
\]
and we then obtain 
\begin{eqnarray}
& & F_d(n_2) + F_d(n_1) \nonumber \\
& = & \left\{ 
  i(F_{d-1}([n_2]^{d-1}_{\Delta}) + [n_2]^{d-1}_{\Delta}) +
  (\lfloor\sqrt[d]{n_2}\rfloor -i)F_{d-1}([n_2]^{d-1}_{\Delta}) + 
  (\lfloor\sqrt[d]{n_2}\rfloor -i-1)[n_2]^{d-1}_{\Delta} + 
  F_{d-1}(n_2') + n_2'
\right\} \nonumber \\
&  & + F_d([n_1]^d_{-}) + F_{d-1}(n_1') + n_1' \nonumber \\
 & = & 
\left\{
  iF_{d-1}([n_2]^{d-1}_{\Delta}) + F_{d-1}(n_2') +  F_{d-1}(n_1')
\right\} + 
\left\{
  i[n_2]^{d-1}_{\Delta} + n_2' + n_1'
\right\} + F_d([n_1]^d_{-}) + G(n_2,i)
\label{eqn:last-expression}
\end{eqnarray}
where 
\[
G(n_2,i):=  
(\lfloor\sqrt[d]{n_2}\rfloor -i)F_{d-1}([n_2]^{d-1}_{\Delta}) + 
(\lfloor\sqrt[d]{n_2}\rfloor -i-1)[n_2]^{d-1}_{\Delta}.
\]
Note that $G(n_2,i)$ can be interpreted geometrically
as the number of edges induced by a rectangular ``box'' in ${\nats}^d$ 
of with base $[n_2]^{d-1}_{\Delta}$ and height $\lfloor\sqrt[d]{n_2}\rfloor -i$,
that is, a box consisting of $\lfloor\sqrt[d]{n_2}\rfloor -i$ copies
of $(d-1)$-cubicles of order $[n_2]^{d-1}_{\Delta}$ stacked one on top
of the other.

By our induction hypothesis, then $F_{d-1} = E_{d-1}$ in 
(\ref{eqn:last-expression}), and recall that 
$i[n_2]^{d-1}_{\Delta} + n_2' + n_1' = [n_1]^{d-1}_{\Delta}+n''$, so we obtain
\begin{eqnarray}
F_d(n_2) + F_d(n_1) & = & 
\left\{
  iE_{d-1}([n_2]^{d-1}_{\Delta}) + E_{d-1}(n_2') +  E_{d-1}(n_1')
\right\} \nonumber \\
\label{eqn:n1n2}
 &   &  + \left\{ [n_1]^{d-1}_{\Delta}+n'' \right\} + F_d([n_1]^d_{-}) + G(n_2,i).
\end{eqnarray}

If $i=0$, we get by induction hypothesis that
\[
iE_{d-1}([n_2]^{d-1}_{\Delta}) + E_{d-1}(n_2') +  E_{d-1}(n_1')
 =     E_{d-1}(n_2') +  E_{d-1}(n_1') 
  \leq E_{d-1}([n_1]^{d-1}_{\Delta}) + E_{d-1}(n'').
\]

If $i>0$, we get since $E_{d-1}$ is super-additive that
\[
iE_{d-1}([n_2]^{d-1}_{\Delta}) + E_{d-1}(n_2') +  E_{d-1}(n_1')\leq 
E_{d-1}((i-1)[n_2]^{d-1}_{\Delta} + n_2' + n_1') + E_{d-1}([n_2]^{d-1}_{\Delta}),
\]
which by inductive hypothesis is 
$\leq E_{d-1}([n_1]^{d-1}_{\Delta}) + E_{d-1}(n'')$.
With this in mind, and that $E_{d-1} = F_{d-1}$,
we obtain from (\ref{eqn:n1n2}) that 
\begin{equation}
\label{eqn:n1n''}
F_d(n_1) + F_d(n_2) \leq 
F_d([n_1]^d_{-}) +
F_{d-1}([n_1]^{d-1}_{\Delta}) + 
[n_1]^{d-1}_{\Delta} +
F_{d-1}(n'') +
n'' +
G(n_2,i).
\end{equation}
Note that by definition of $n''$ we have
$i[n_2]^{d-1}_{\Delta} + n_2' + n_1' = [n_1]^{d-1}_{\Delta}+n''$,
and since $n_1+n_2 = [n_1]^d_{+} + n'$, we have 
$n' = (\lfloor\sqrt[d]{n_2}\rfloor -i)[n_2]^{d-1}_{\Delta} + n''$.
We now consider two cases.

{\sc First case $i=\lfloor\sqrt[d]{n_2}\rfloor$:}
In this case 
$G(n_2,i) = G(n_2,\lfloor\sqrt[d]{n_2}\rfloor) = -[n_2]^{d-1}_{\Delta}$,
and (\ref{eqn:n1n''}) becomes
\begin{eqnarray*}
F_d(n_1) + F_d(n_2) 
  & \leq &  
\left\{
  F_d([n_1]^d_{-}) +
  F_{d-1}([n_1]^{d-1}_{\Delta}) + 
  [n_1]^{d-1}_{\Delta} +
\right\} +
F_{d-1}(n'') +
n'' - [n_2]^{d-1}_{\Delta} \\ 
  & =    & F_d([n_1]^d_{+}) + F_{d-1}(n'') + n'' - [n_2]^{d-1}_{\Delta} \\
  & \leq & F_d([n_1]^d_{+}) + F_{d-1}(n''), 
\end{eqnarray*}
which is $\leq F_d([n_1]^d_{+}) + F_d(n'')$, since 
$F_{d-1}(n'') = E_{d-1}(n'') \leq E_d(n'') = F_d(n'')$ by induction
hypothesis. This proves $\mathbf{P}'(d,n)$ in this case 
since $n' = n''$.

{\sc Second case $i < \lfloor\sqrt[d]{n_2}\rfloor$:}
Since $n''< [n_2]^{d-1}_{\Delta}$, one can put $n''$ points
on one $[n_2]^{d-1}_{\Delta}$-side of the $G(n_2,i)$-box mentioned
here above, thereby obtaining $G(n_2,i) + F_{d-1}(n'') + n''$ edges.
Hence we have 
\[
G(n_2,i) + F_{d-1}(n'') + n'' \leq 
E_d((\lfloor\sqrt[d]{n_2}\rfloor -i)[n_2]^{d-1}_{\Delta} + n''), 
\]
and (\ref{eqn:n1n''}) yields
\begin{eqnarray*}
F_d(n_1) + F_d(n_2) 
  & \leq &  
\left\{
  F_d([n_1]^d_{-}) +
  F_{d-1}([n_1]^{d-1}_{\Delta}) + 
  [n_1]^{d-1}_{\Delta} +
\right\} +
E_d((\lfloor\sqrt[d]{n_2}\rfloor -i)[n_2]^{d-1}_{\Delta} + n'') \\
  & =    & F_d([n_1]^d_{+}) + 
F_d((\lfloor\sqrt[d]{n_2}\rfloor -i)[n_2]^{d-1}_{\Delta} + n'')
\end{eqnarray*}
which proves $\mathbf{P}'(d,n)$ in this case as well, 
since $n' = (\lfloor\sqrt[d]{n_2}\rfloor -i)[n_2]^{d-1}_{\Delta} + n''$ 
here in this case. This completes the proof of (\ref{eqn:1st-impl}).
We complete the proof of (\ref{eqn:main-impl}) in the following
section.

\section{The final steps in the proof of $E = F$}
\label{sec:E=F}

In this section we prove the following implication
\[
\mathbf{P}(d,n)\wedge\mathbf{EF}(d,n)\mbox{ for }d+n<N
\Rightarrow
\mathbf{P}'(d,n)\wedge\mathbf{EF}(d,n)\mbox{ for }d+n = N.
\]
From previous section we already have (\ref{eqn:1st-impl}), so it
suffices to verify the following implication
\begin{equation}
\label{eqn:E=F}
\mathbf{P}(d,n)\wedge\mathbf{EF}(d,n)\mbox{ for }d+n<N
\Rightarrow
\mathbf{EF}(d,n)\mbox{ for }d+n = N.
\end{equation}
Before we delve into that, we need a property
of fully nested sets in general.

Let $S\subseteq {\nats}^d$ be a fully nested set with $|S|=n$.
For each $i\in [d]$ let $h_i = |\pi_i(S)|$ be the {\em height} 
of $S$ along $i$-th axis, and let  
$A_i = |\{ \tilde{x}\in S : x_i = h_i\}|$ the {\em area}
of the top layer of $S$ along the $i$-th axis.
Since $S$ is fully nested we have $n = |S| \geq h_iA_i$
for each $i$ and hence 
\begin{eqnarray*}
\frac{\min(A_1,\ldots,A_d)}{n} & =    & 
  \min\left(\frac{A_1}{n},\ldots, \frac{A_d}{n}\right) \\
  & \leq &
\min\left(\frac{A_1}{h_1A_1},\ldots, \frac{A_d}{h_dA_d}\right) \\
  & =    & \min\left(\frac{1}{h_1},\ldots,\frac{1}{h_d}\right) \\
  & =    & \frac{1}{\max(h_1,\ldots,h_d)}.
\end{eqnarray*}
By Proposition~\ref{prp:n-ass-tight} we can assume that 
$n$ is not a $d$-th power of an integer. Hence, if $n$ has $d$-PCR given by
$n = [m_d,{\ell}_d]^d + [m_{d-1},{\ell}_{d-1}]^{d-1} 
+ \cdots + [m_{c},{\ell}_{c}]^{c}$, then either
$c<d$, or $c=d$ and $1\leq {\ell}_d\leq d-1$. In either case we 
have $m_d < \sqrt[d]{n} < m_d+1$. By definition of $h_i$ we have 
\[
m_d^d < n < h_1\cdots h_d 
\]
and hence there is at least one $i$ with $h_i \geq m_d+1$,
and so $\max(h_1,\ldots,h_d) \geq m_d+1$.
From above we then have
\[
\frac{\min(A_1,\ldots,A_d)}{n} \leq \frac{1}{m_d+1}.
\]
By symmetry we may assume $\min(A_1,\ldots,A_d) = A_d$
and so $A_d \leq \frac{n}{m_d+1}$.
From this we get the following.
\begin{observation}
\label{obs:n2-small}
Let $S\subseteq {\nats}^d$ be a fully nested set with $|S|=n$.
Then there is a suitable permutation of the coordinates of $S$ 
and a partition $S = S_1\cup S_2$ such that 
(\ref{eqn:rec-exact2})
holds and where $n_2 = |S_2| = A_d \leq \frac{n}{m_d+1}$.
\end{observation}
So if $S\subseteq {\nats}^d$ is a fully nested set with $|S|=n$
and where the $d$-PCR is given by
$n = [m_d,{\ell}_d]^d + [m_{d-1},{\ell}_{d-1}]^{d-1} 
+ \cdots + [m_{c},{\ell}_{c}]^{c} = [n]^d_{-} + n'$, then
by~\ref{obs:n2-small} we can
assume there is partition $S = S_1\cup S_2$ with
\[
n_1 = n - n_2 \geq \frac{m_d}{m_d+1}n \geq \frac{m_d}{m_d+1}[n]^d_{-}
\geq [m_d,{\ell}_d-1]^d.
\]
Note, with our convention that $[m_d,0]^d = [m_d-1,d]^d$ and keeping
in mind that $m_d/(m_d+1)\geq (m_1-1)/m_d$, the above
inequality is valid also for ${\ell}_d = 0$. We therefore have 
the following.
\begin{observation}
\label{obs:n1-2poss}
Let $S\subseteq {\nats}^d$ be a fully nested set with $|S|=n$,
where the $d$-PCR is given by
$n = [m_d,{\ell}_d]^d + [m_{d-1},{\ell}_{d-1}]^{d-1} 
+ \cdots + [m_{c},{\ell}_{c}]^{c} = [n]^d_{-} + n'$,
then we can assume 
there is a partition $S = S_1\cup S_2$ such that 
(\ref{eqn:rec-exact2}) holds and where 
$[n_1]^d_{-} = [n]^d_{-}$ or $[n_1]^d_{-} = [n]^d_{2\/-}$.
\end{observation}
Consider now an integer partition $n = n_1+n_2$
where $n = [m_d,{\ell}_d]^d + [m_{d-1},{\ell}_{d-1}]^{d-1} 
+ \cdots + [m_{c},{\ell}_{c}]^{c}$ is the $d$-PCR of $n$ 
and $n_2\leq \frac{n}{m_d+1}$. We then have
$n_1\geq \frac{m_d}{m_d+1}n$ and we have the following.
\begin{lemma}
\label{lmm:attached}
If $n=n_1+n_2$ where $n_1\geq \frac{m_d}{m_d+1}n$ and
$n_2\leq \frac{n}{m_d+1}$, then $\llbracket n_1\rrbracket^d$ 
has a side that covers $\llbracket n_2\rrbracket^{d-1}$.
\end{lemma}
\begin{proof}
Note that $n=n_1+n_2$ where $n_2\leq \frac{n}{m_d+1}$
is equivalent to $n=n_1+n_2$ where $n_2\leq n_1/m_d$.
Let the $d$-PCR of $n$ be given by
$n = [m_d,{\ell}_d]^d + [m_{d-1},{\ell}_{d-1}]^{d-1} 
+ \cdots + [m_{c},{\ell}_{c}]^{c}$. By Observation~\ref{obs:n1-2poss}
we have two cases to consider.

{\sc First case:} $[n_1]^d_{-} = [n]^d_{-} = [m_d,{\ell}_d]^d$.
Here $\llbracket n_1\rrbracket^d$ has a side 
$S = \llbracket m_d,{\ell}_d\rrbracket^{d-1}$ and since
$n < [m_d,{\ell}_d+1]^d$ we have $|S| \geq n/(m_d+1)$.
Since $n_2 \leq n/(m_d+1) \leq |S|$, then $S$ covers 
$\llbracket n_2\rrbracket^{d-1}$ by Proposition \ref{prp:cubicle}.

{\sc Second case:} $[n_1]^d_{-} = [n]^d_{2\/-} = [m_d,{\ell}_d-1]^d$.
Here we need to consider two possibilities:

If ${\ell}_d\geq 1$, then $0\leq {\ell}_d-1<d-1$, so by
Lemma~\ref{lmm:side} $\llbracket n_1\rrbracket^d$ has a side
$S$ with $|S|\geq n_1/m_d$. Since $n_2\leq n_1/m_d$ 
then $S$ covers 
$\llbracket n_2\rrbracket^{d-1}$ by Proposition \ref{prp:cubicle}.

If ${\ell}_d = 0$, then $[n_1]^d_{-} = [m_d-1,d-1]^d$, and 
 $\llbracket n_1\rrbracket^d$ has a (bottom) side
$S = \llbracket m_d-1,d-1\rrbracket^{d-1} = \llbracket m_d,0\rrbracket^{d-1}$. 
Since $[m_d,0]^d > n_1$ we have $|S| > n_1/m_d$, so as before
$S$ covers 
$\llbracket n_2\rrbracket^{d-1}$ by Proposition \ref{prp:cubicle}.
\end{proof}

We now consider the two cases based on the above Observation~\ref{obs:n1-2poss}.

{\sc First case:} $[n_1]^d_{-} = [n]^d_{-} = [m_d,{\ell}_d]^d$.
In this case we have for our partition $S = S_1\cup S_2$ that
$n_1\geq [n]^d_{-}$ and hence $n_1 =  [n]^d_{-} + n_1'$
where $n_1'\geq 0$ and 
$n_2 \leq n - [n]^d_{-}$. 
Since $n_2 = n- n_1$, we have by (\ref{eqn:rec-exact2}) and 
induction hypothesis that
\begin{eqnarray}
E_d(n) & =    & E_d(S_1) + E_{d-1}(S_2) + n_2 \nonumber \\
       & \leq &  F_d(n_1) + F_{d-1}(n_2) + n_2 \nonumber \\
       & =    & F_d([n]^d_{-}+n_1') 
                + F_{d-1}(n_2) + n_2 \label{eqn:alpha}
\end{eqnarray}
Here the $d$-PCR of $n_1$ is obtained
by adding $[n]^d_{-}$ to the $d$-PCR of $n_1'$.
With this in mind we obtain by (\ref{eqn:F(n)-rec}) that
\[
F_d([n]^d_{-}+n_1') = F_d([n]^d_{-})
+ F_{d-1}(n_1') + n_1'.
\]
Substituting this expression into (\ref{eqn:alpha})
we then obtain
\[
E_d(n) \leq  F_d([n]^d_{-}) 
           + F_{d-1}(n_1') + F_{d-1}(n_2)
                + n-[n]^d_{-}.
\]
By induction hypothesis and the super-additivity of $E_{d-1}$ we obtain
\begin{eqnarray*}
F_{d-1}(n_1') + F_{d-1}(n_2) 
  & =    & E_{d-1}(n_1') + E_{d-1}(n_2) \\ 
  & \leq & E_{d-1}(n-[n]^d_{-})) \\
  & =    & F_{d-1}(n-[n]^d_{-}),
\end{eqnarray*}
and hence by (\ref{eqn:F(n)-rec}) that 
\[
E_d(n) \leq  F_d([n]^d_{-}) 
 + F_{d-1}(n-[n]^d_{-}) +  n-[n]^d_{-} = F_d(n),
\]
which, by definition of $E_d(n)$, proves (\ref{eqn:E=F}) in this case.

{\sc Second case:} $[n_1]^d_{-} = [n]^d_{2\/-} = [m_d,{\ell}_d-1]^d$ 
(Recall, if ${\ell}_d=0$
then $[m_d,0]^d = [m_d-1,d]^d$ and hence $[n_1]^d_{-} = [m_d-1,d-1]^d$.)
In this case we have for our partition $S = S_1\cup S_2$ that
$[n]^d_{2\/-} \leq n_1 < [n]^d_{-} = [m_d,{\ell}_d]^d$ and
hence $n_1 = [n_1]^d_{-} + n_1'$, where
$0\leq n_1' < [n]^d_{-}-[n]^d_{2\/-} = [m_d,{\ell}_d-1]^{d-1}$, a difference
of two consecutive pseudo $d$-cubics, and hence itself a 
pseudo $(d-1)$-cubic.
By (\ref{eqn:rec-exact2}) and by induction hypothesis 
we have, as in previous case, that
\begin{eqnarray}
E_d(n) & =    & E_d(S_1) + E_{d-1}(S_2) + n_2 \nonumber \\
       & \leq & F_d(n_1) + F_{d-1}(n_2) + n_2 \nonumber \\
       & =    & F_d([n_1]^d_{-} + n_1') + F_{d-1}(n_2) + n_2.\label{eqn:beta}
\end{eqnarray}
As in the previous case the $d$-PCR of $n_1$ is obtained by
adding the $d$-PCR of $n_1'$ to $[n_1]^d_{-}$.
With this in mind we obtain, as before, by (\ref{eqn:F(n)-rec}) that
\[
F_d([n_1]^d_{-}+n_1') = F_d([n_1]^d_{-})
+ F_{d-1}(n_1') + n_1'.
\]
Substituting this expression into (\ref{eqn:beta})
we then obtain
\begin{equation}
\label{eqn:Eleq}
E_d(n) \leq  F_d([n_1]^d_{-}) 
           + F_{d-1}(n_1') + F_{d-1}(n_2)
                + (n - [n]^d_{2\/-}).
\end{equation}
By (\ref{eqn:F(n)-rec}) and  (\ref{eqn:F(pc)-rec}) we have
\begin{equation}
\label{eqn:Feq}
F_d(n) = F_d([n]^d_{2\/-})  
       + F_{d-1}([n]^d_{-}-[n]^d_{2\/-}) + F_{d-1}(n-[n]^d_{-})
       + (n - [n]^d_{2\/-}).
\end{equation}
By (\ref{eqn:Eleq}) and (\ref{eqn:Feq}) we see that $E_d(n)\leq F_d(n)$
can be obtained from the following inequality
\begin{equation}
\label{eqn:2leq2}
 F_{d-1}(n_1') + F_{d-1}(n_2) 
\leq F_{d-1}([n]^d_{-}-[n]^d_{2\/-}) + F_{d-1}(n-[n]^d_{-}).
\end{equation}
We have $n_1' < [n]^d_{-}-[n]^d_{2\/-}$. If
also $n_2 < [n]^d_{-}-[n]^d_{2\/-} $, then (\ref{eqn:2leq2}) follows
from our inductive hypothesis $\mathbf{P}(d-1,n_1'+n_2)$ and so 
we have  (\ref{eqn:E=F}) in this case.

Otherwise we have $n_2 > [n]^d_{-}-[n]^d_{2\/-} = [m_d,{\ell}_d-1]^{d-1}$.
Writing $n_2 = [m_d,{\ell}_d-1]^{d-1} + n_2''$ and noting that
$n_2\leq n/(m_d+1)$ we obtain 
\[
n_2'' \leq \frac{n}{m_d+1} - [m_d,{\ell}_d-1]^{d-1} = \frac{n'}{m_d+1}
< \frac{[m_d,{\ell}_d]^{d-1}}{m_d+1} = [m_d,{\ell}_d-1]^{d-2},
\]
and hence we have 
$[n_2]_{-}^{d-1} = [m_d,{\ell}_d-1]^{d-1} = [n]_{-}^d - [n]_{2\/-}^d$.
By (\ref{eqn:F(n)-rec}) we then have
\[
F_{d-1}(n_2) = F_{d-1}([n_2]_{-}^{d-1}) + F_{d-2}(n_2'') + n_2'',
\]
and hence $E_d(n)\leq F_d(n)$, which can be obtained from 
(\ref{eqn:2leq2}), can therefore be obtained from 
\[
F_{d-1}(n_1') + F_{d-2}(n_2'') + n_2'' \leq F_{d-1}(n').
\]
By induction hypothesis we have $F_{d-1}(n') = E_{d-1}(n')$, and so
(\ref{eqn:2leq2}) is valid if 
\begin{equation}
\label{eqn:n_1'n_2''}
F_{d-1}(n_1') + F_{d-2}(n_2'') + n_2'' \leq E_{d-1}(n').
\end{equation}
Interpreting the quantity on the left of (\ref{eqn:n_1'n_2''})
as the number of edges of induced by a set $S'\subseteq {\nats}^{d-1}$
with $|S'| = n' = n-[n]_{-}^d$, we see it can be realized if
$\llbracket n_1'\rrbracket^{d-1}$ has a side that covers 
$\llbracket n_2''\rrbracket^{d-2}$. In that case we have
\[
F_{d-1}(n_1') + F_{d-2}(n_2'') + n_2'' = F_{d-1}(S')\leq  E_{d-1}(n'),
\]
where $S'\subseteq {\nats}^{d-1}$ is obtained by attaching 
$\llbracket n_2''\rrbracket^{d-2}$ to one of the sides of 
$\llbracket n_1'\rrbracket^{d-1}$. We will now verify
this, and thereby completing the inductive step of (\ref{eqn:E=F}).

Recall the $d$-PCR $n = [m_d,{\ell}_d]^d + [m_{d-1},{\ell}_{d-1}]^{d-1} 
+ \cdots + [m_{c},{\ell}_{c}]^{c} = [n]^d_{-} + n'$.
Since in this case we have $[n_1]^d_{-} = [n]^d_{2\/-} = [m_d,{\ell}_d-1]^d$
and $[n_2]^{d-1}_{-} = [n]^d_{-}-[n]^d_{2\/-} = [m_d,{\ell}_d-1]^{d-1}$ we have
\[
n_1' \geq \frac{m_d}{m_d+1}n - [n]^d_{2\/-} = \frac{m_d}{m_d+1}n',
\]
and since $n_1' + n_2'' = n'$ we also get
\[
n_2'' = n' - n_1' \leq n' - \frac{m_d}{m_d+1}n' = \frac{n'}{m_d+1}.
\]
Since $m_{d-1} \leq m_d$ and the $(d-1)$-PCR of $n'$
is given by $n' = [m_{d-1},{\ell}_{d-1}]^{d-1} + \cdots + [m_{c},{\ell}_{c}]^{c}$,
that $\llbracket n_1'\rrbracket^{d-1}$ has a side that covers
$\llbracket n_2''\rrbracket^{d-2}$ now follows from 
Lemma~\ref{lmm:attached}. This completes the proof of (\ref{eqn:n_1'n_2''})
and hence (\ref{eqn:2leq2}), which then completes the proof
of (\ref{eqn:E=F}) in this case. -- This completes the inductive
proof of $\mathbf{EF}(d,n)$ for all $d,n\in\nats$. By 
Observation~\ref{obs:F-formula} we have the following summarizing 
theorem, the main theorem of this article.
\begin{theorem}
\label{thm:main}
For $n\in\nats$ with the $d$-PCR $n = [m_d,{\ell}_d]^d + 
[m_{d-1},{\ell}_{d-1}]^{d-1} + \cdots + [m_{c},{\ell}_{c}]^{c}$
we have that the maximum number $E_d(n)$ of edges a set $S\subseteq {\nats}^d$
with $|S|=n$ can induce is given by $E_d(n) = dn - \delta_d(n)$
where the discrepancy is given by 
\[
\delta_d(n) = \sum_{i=c}^d
\left(\ell_i[m_i,\ell_i-1]^{i-1} + (i-\ell_i)[m_i,\ell_i]^{i-1}\right).
\]
\end{theorem}

\section{Some final observations and corollaries}
\label{sec:some-obs}

\paragraph{An inequality on projections}
Let $S\subseteq {\nats}^d$ be a fully nested set
where $|S| = n = [m,{\ell}]^d$ is a pseudo $d$-cubic.
In this case we have by Theorem~\ref{thm:main}
that 
\[
E_d(n) = E_d([m,{\ell}]^d) = dn - 
{\ell}[m,\ell-1]^{d-1} + (d-{\ell})[m,{\ell}]^{d-1}.
\]
By Observation~\ref{obs:exact-edge-count} we then have
that 
\[
E_d(S) = dn - (n_{\hat{1}}(S) + \cdots + n_{\hat{d}}(S)) 
\leq dn - {\ell}[m,\ell-1]^{d-1} + (d-{\ell})[m,{\ell}]^{d-1},
\]
and hence for any fully nested set $S\subseteq {\nats}^d$
with $|S| = [m,{\ell}]^d$ we have
\begin{equation}
\label{eqn:sum-nested}
n_{\hat{1}}(S) + \cdots + n_{\hat{d}}(S) \geq
{\ell}[m,\ell-1]^{d-1} + (d-{\ell})[m,{\ell}]^{d-1}.
\end{equation}
Recall $g_i$, the gravity along $i$-th axis, from 
Section~\ref{sec:f-nested}. 
For an arbitrary set $T\subseteq {\nats}^2$ 
we have $|\pi_2(g_1(T))| =|\pi_2(T)|$
and  
\[
|\pi_1(g_1(T))| = \max\{|T\cap \pi_2^{-1}(x)|:x\in\nats\}\leq |\pi_1(T)|.
\]
Hence, if $S\subseteq {\nats}^d$ is an arbitrary set (not necessarily
fully nested) then
$n_{\hat{\imath}}(S) = n_{\hat{\imath}}(g_i(S))$
and for $j\neq i$ we have a partition
\[
S = \bigcup_{\tilde{x}\in{\nats}^{d-2}} S_{\tilde{x}},
\]
where $S_{\tilde{x}} = S\cap \pi_{\hat{\imath},\hat{\jmath}}^{-1}(\tilde{x})$,
and we then obtain
\begin{eqnarray*}
n_{\hat{\imath}}(S) & = & 
\left| \pi_{\hat{\imath}}\left( 
  \bigcup_{\tilde{x}\in{\nats}^{d-2}} S_{\tilde{x}} \right) \right| \\
  & =    & \sum_{\tilde{x}\in{\nats}^{d-2}} 
  |\pi_{\hat{\imath}}(S_{\tilde{x}})| \\
  & \geq & \sum_{\tilde{x}\in{\nats}^{d-2}} 
  |\pi_{\hat{\imath}}(g_j(S_{\tilde{x}}))| \\ 
  & =    & 
\left| \pi_{\hat{\imath}}\left( g_j\left( 
  \bigcup_{\tilde{x}\in{\nats}^{d-2}} S_{\tilde{x}}\right) \right) \right| \\
  & =    & 
|\pi_{\hat{\imath}}(g_j(S))|.
\end{eqnarray*}
Since $n_{\hat{\imath}}(S)\geq |\pi_{\hat{\imath}}(g_j(S))|$
for all $i,j\in \{1,\ldots,d\}$ we then obtain for the 
total gravity $g = g_1g_2\cdots g_d$ that
$n_{\hat{\imath}}(S) = |\pi_{\hat{\imath}}(S)|
\geq |\pi_{\hat{\imath}}(g(S))|$.
From this and (\ref{eqn:sum-nested}) we therefore 
we have the following corollary that relates the
cardinality of a point set of ${\nats}^d$ to that of its
projections, in the spirit of Theorem 2 of Bollob\'{a}s and 
Thomason~\cite[Thm 2, p.418]{Bollobas-Thomason}.
\begin{corollary}
\label{cor:linear-proj}
For an arbitrary set $S\subseteq {\nats}^d$ with 
$|S| = [m,{\ell}]^d$, we have
\[
n_{\hat{1}}(S) + \cdots + n_{\hat{d}}(S) \geq
{\ell}[m,\ell-1]^{d-1} + (d-{\ell})[m,{\ell}]^{d-1}.
\]
\end{corollary}
The above corollary can, of course, be generalized to
an inequality for a general $n\in\nats$ in terms
of its $d$-PCR, although  the formula will be more complicated.

\paragraph{The case of $d=2$} 
Any $n\in\nats$ has a $2$-PCR given by 
$n = [m_2,{\ell}_2]^2 + [m_1,0]^1$ where ${\ell}_2 \in\{0,1\}$
and $m_1 = [m_1,0]^1 < [m_2,{\ell}_2]^1$, and 
so $m_1 < [m_2,1]^1 = m_2+1$ if ${\ell}_2=1$ and
$m_1 < [m_2,0]^1 = m_2$ if ${\ell}_2 = 0$.
By Theorem~\ref{thm:main} we have that $E_2(n) = 2n - \delta_2(n)$,
where
\[
\delta_2(n) = 
{\ell}_2[m_2,{\ell}_2-1]^1 + (2-{\ell}_2)[m_2,{\ell}_2]^1 + [m_1,0]^0 = 
{\ell}_2[m_2,{\ell}_2-1]^1 + (2-{\ell}_2)[m_2,{\ell}_2]^1 + 1,
\]
and hence $\delta_2(n) = 2m_2 + 1$ if ${\ell}_2 = 0$ and 
$\delta_2(n) = 2m_2 + 2$ if ${\ell}_2 = 1$. From this we see
that $\delta_2(n) = \left\lceil 2\sqrt{n}\right\rceil$, which agrees
with the formula $E_2(n) = \lfloor 2n - 2\sqrt{n}\rfloor$ given
in~\cite{Harary-Harborth}.

\paragraph{The case $n < 2^d$} In this case 
the $d$-PCR of $n$ has the form
$n = [m_d,{\ell}_d]^d + [m_{d-1},{\ell}_{d-1}]^{d-1} + 
\cdots + [m_{c},{\ell}_{c}]^{c}$ where each $m_i = 1$ and
$d-1 \geq {\ell}_d > {\ell}_{d-1} > \cdots > {\ell}_c \geq 0$,
which is exactly the usual binary representation of 
$n = 2^{{\ell}_d} + 2^{{\ell}_{d-1}} + \cdots + 2^{{\ell}_c}$.
By Theorem~\ref{thm:main} we have that $E_d(n) = dn - \delta_d(n)$,
where
\[
\delta_d(n) =  
\sum_{i=c}^d({\ell}_i2^{{\ell}_i-1} + (i-{\ell}_i)2^{{\ell}_i}) = 
\sum_{i=c}^d2^{{\ell}_i-1}(2i-{\ell}_i).
\]
and hence 
\begin{equation}
\label{eqn:Ed-sumof1}
E_d(n) = \sum_{i=c}^d2^{{\ell}_i-1}(2d + {\ell}_i - 2i).
\end{equation}
Now, since $n<2^d$, the maximum number of edges 
a set $S\subseteq {\nats}^d$ of order $n$ can induce,
is the same as the maximum number of edges a set $S$ of order
$n$ in  any rectangular grid can induce. So, $E_d(n) = f(n)$ where
$f(n)$ is the total number of 1s in the binary representation
of $1,\ldots,n-1$, as first proved in~\cite{Mcllroy-SIAM-binary}
and also stated in~\cite{Geir-Hcube}[Obs.~1.2].
The sequence 
$(f(n))_1^{\infty} = (0,1,2,4,5,7,9,12,13,15,17,20,22,25,28,32,\ldots)$ 
is well known~\cite[A000788]{sum-of-1s},
and has appears naturally when analysing worst-case scenarios 
in sorting algorithms. It has been studied extensively
in a variety of papers, as discussed in detail in~\cite{Geir-Hcube}, 
as it is one of the very few exact known solutions
to a common divide-and-conquer recurrence 
relation~\cite{Geir-Hcube}[Obs.~1.2].

From (\ref{eqn:Ed-sumof1}) we have the following alternative explicit 
formula for $f(n)$ in terms of the binary representation of $n$.
\begin{corollary}
\label{cor:sumof1-formula}
For any $n\in\nats$ with binary representation given
by $n = 2^{{\ell}_d} + 2^{{\ell}_{d-1}} + \cdots + 2^{{\ell}_c}$,
the total number of 1s appearing in the binary representations of 
$1,\ldots,n-1$ is given by 
\[
f(n) = \sum_{i=c}^d2^{{\ell}_i-1}(2d + {\ell}_i - 2i).
\]
\end{corollary}

\subsection*{Acknowledgments}  

Geir Agnarsson wants to thank Kshitij Lauria for 
contacting him via email, asking pointed and interesting questions,
and for keeping this project alive.

\flushright{\today}

\end{document}